%% file: SVRG_arXiv_Submission_v1.tex
\documentclass{article}

\input{FancyArticlePackages}

\input{FancyMacros}

\input{FancyMathAndTheoremEnvironments}

\usepackage{csquotes}
\usepackage[american]{babel}
\usepackage[backend=biber,citestyle=authoryear,bibstyle=authortitle,arxiv=abs,url=true,natbib]{biblatex}

\newbibmacro{string+doiurlisbn}[1]{%
  \iffieldundef{url}{%
    \iffieldundef{doi}{%
      \iffieldundef{isbn}{%
        \iffieldundef{issn}{%
          #1%
        }{%
          \href{http://books.google.com/books?vid=ISSN\thefield{issn}}{#1}%
        }%
      }{%
        \href{http://books.google.com/books?vid=ISBN\thefield{isbn}}{#1}%
      }%
    }{%
      \href{http://dx.doi.org/\thefield{doi}}{#1}%
    }%
  }{%
    \href{\thefield{url}}{#1}%
  }%
}

\DeclareFieldFormat{title}{\usebibmacro{string+doiurlisbn}{\mkbibemph{#1}}}
\DeclareFieldFormat[article,incollection,inproceedings]{title}%
    {\usebibmacro{string+doiurlisbn}{\mkbibquote{#1}}}

\DeclareCiteCommand{\cite}
  {\usebibmacro{prenote}}
  {\usebibmacro{citeindex}%
   \printtext[bibhyperref]{\usebibmacro{cite}}}
  {\multicitedelim}
  {\usebibmacro{postnote}}

\DeclareCiteCommand{\parencite}[\mkbibparens]
  {\usebibmacro{prenote}}
  {\usebibmacro{citeindex}%
    \printtext[bibhyperref]{\usebibmacro{cite}}}
  {\multicitedelim}
  {\usebibmacro{postnote}}

\newcommand{\eps}{\epsilon}

\usepackage{wrapfig}
\usepackage{algorithm,algpseudocode}
\algrenewcommand\algorithmicindent{1.0em}

\begin{document}
\title{Breaking the Span Assumption Yields Fast Finite-Sum Minimization}

\author[1]{Robert Hannah\thanks{Corresponding author: \href{mailto:RobertHannah89@gmail.com}{RobertHannah89@gmail.com}}}
\author[1]{Yanli Liu\thanks{\href{mailto:yanli@math.ucla.edu}{yanli@math.ucla.edu}}}
\author[2]{Daniel O'Connor\thanks{\href{mailto:daniel.v.oconnor@gmail.com}{daniel.v.oconnor@gmail.com}}}
\author[1]{Wotao Yin\thanks{\href{mailto:WotaoYin@math.ucla.edu}{WotaoYin@math.ucla.edu}}}
\affil[1]{Department of Mathematics, University of California, Los Angeles}
\affil[2]{Department of Radiation Oncology, University of California, Los Angeles}
\setcounter{Maxaffil}{0}
\renewcommand\Affilfont{\itshape\small}

\date{\today}

\maketitle

\begin{abstract}
In this paper, we show that SVRG and SARAH can be modified to be fundamentally faster than all of the other standard algorithms that minimize the sum of $n$ smooth functions, such as SAGA, SAG, SDCA, and SDCA without duality. Most finite sum algorithms follow what we call the ``span assumption'': Their updates are in the span of a sequence of component gradients chosen in a random IID fashion. In the big data regime, where the condition number $\kappa=\mathcal{O}(n)$, the span assumption prevents algorithms from converging to an approximate solution of accuracy $\epsilon$ in less than $n\ln(1/\epsilon)$ iterations. SVRG and SARAH do not follow the span assumption since they are updated with a hybrid of full-gradient and component-gradient information. We show that because of this, they can be up to $\Omega(1+(\ln(n/\kappa))_+)$ times faster. In particular, to obtain an accuracy $\epsilon = 1/n^\alpha$ for $\kappa=n^\beta$ and $\alpha,\beta\in(0,1)$, modified SVRG requires $\mathcal{O}(n)$ iterations, whereas algorithms that follow the span assumption require $\mathcal{O}(n\ln(n))$ iterations. Moreover, we present lower bound results that show this speedup is optimal, and provide analysis to help explain why this speedup exists. With the understanding that the span assumption is a point of weakness of finite sum algorithms, future work may purposefully exploit this to yield even faster algorithms in the big data regime.

\end{abstract}

\section{Introduction}
Finite sum minimization is an important class of optimization problem that appears
in many applications in machine learning and other areas. We consider
the problem of finding an approximation $\hat{x}$ to the minimizer $x^{*}$
of functions $F:\RR^{d}\to\RR$ of the form:
\begin{align}
F\p x & =f(x)+\psi(x)=\frac{1}{n}\sum_{i=1}^{n}f_{i}\p x +\psi\p{x}.\label{eq:Average-of-fi}
\end{align}
We assume each function $f_{i}$ is smooth\footnote{A function $f$ is $L$-smooth if it has an $L$-Lipschitz gradient $\nabla f$}, and possibly nonconvex; $\psi$ is proper, closed, and convex; and
the sum $F$ is strongly convex and smooth. It has become well-known that
under a variety of assumptions, functions of this form can be
minimized much faster with variance reduction (VR) algorithms that
specifically exploit the finite-sum structure. When each $f_{i}$
is $\mu$-strongly convex and $L$-smooth, and $\psi=0$, SAGA \parencite{DefazioBachLacoste-Julien2014_saga},
SAG \parencite{RouxSchmidtBach2012_stochastic}, Finito/Miso \parencite{DefazioDomkeCaetano2014_finito,Mairal2013_optimization},
SVRG \parencite{JohnsonZhang2013_accelerating}, SARAH \parencite{NguyenLiuScheinbergTakac2017_sarah},
SDCA \parencite{Shalev-ShwartzZhang2013_stochastic}, and SDCA without duality \parencite{Shalev-Shwartz2016_sdca} can find a vector
$\hat{x}$ with expected suboptimality $\EE\p{f\p{\hat{x}}-f\p{x^{*}}}=\cO\p{\eps}$
with only $\cO\p{\p{n+L/\mu}\ln\p{1/\eps}}$ calculations
of component gradients $\nabla f_{i}\p x$. This can be up to $n$ times faster than (full)
gradient descent, which takes $\cO\p{n L/\mu \ln\p{1/\eps}}$ gradients.
These algorithms exhibit sublinear convergence for non-strongly convex
problems\footnote{SDCA must be modified however with a dummy regularizer.}. Various results also exist for nonzero convex $\psi$.

Accelerated VR algorithms have also been proposed. Katyusha \parencite{Allen-Zhu2017_katyusha}
is a primal-only Nesterov-accelerated VR algorithm that uses only component gradients.
It is based on SVRG and has complexity $\cO\p{\p{n+\sqrt{n\kappa}}\ln\p{1/\eps}}$) for condition number $\kappa$ which is defined as $L/\mu$. In \parencite{Defazio2016_simple}, the author devises an accelerated SAGA algorithm that attains the same complexity using component proximal steps. In \parencite{LanZhou2017_optimal}, the author devises an accelerated primal-dual VR algorithm.
There also exist ``catalyst'' \parencite{LinMairalHarchaoui2015_universala}
accelerated methods \parencite{LinLuXiao2014_accelerateda,Shalev-ShwartzZhang2016_accelerated}.
However, catalyst methods appear to have a logarithmic complexity penalty
over Nesterov-accelerated methods.

In \parencite{LanZhou2017_optimal}, authors show that a class of algorithms
that includes SAGA, SAG, Finito (with replacement), Miso, SDCA without duality, etc. have
complexity $K(\epsilon)$ lower bounded by $\Omega\p{\p{n+\sqrt{n\kappa}}\ln\p{1/\eps}}$ for problem dimension $d\geq 2K(\epsilon)$. More precisely, the lower bound applies to algorithms that satisfy what we will call the \textbf{span condition}. That is 
\begin{align}
x^{k+1} & \in x^{0}+\text{span}\cp{\nabla f_{i_{0}}\p{x^{0}},\nabla f_{i_{1}}\p{x^{1}},\ldots,\nabla f_{i_{k}}\p{x^{k}}}\label{eq:SpanCondition}
\end{align}
for some fixed IID random variable $i_k$ over the indices $\cp{1,\ldots,n}$.
Later, \parencite{WoodworthSrebro2016_tight} and \parencite{ArjevaniShamir2016_dimensionfreea}
extend lower bound results to algorithms that do not follow the span assumption: SDCA, SVRG, SARAH, accelerated SAGA, etc.;
but with a smaller lower bound of $\Omega\p{n+\sqrt{n\kappa}\ln\p{1/\eps}}$.
The difference in these two expressions was thought to be a proof artifact
that would later be fixed. 

However we show a surprising result in Section \ref{sec:OptimalSVRG}, that SVRG, and SARAH can be fundamentally
faster than methods that satisfy the span assumption, with the full gradient steps playing a critical role in their speedup. More precisely, for $\kappa=\cO\p n$,
SVRG and SARAH can be modified to reach an accuracy of $\eps$ in
$\cO((\frac{n}{1+(\ln\p{n/\kappa})_+} )\ln\p{1/\eps})$ gradient calculations\footnote{We define $(a)_+$ as $\max\cp{a,0}$ for $a\in\RR$.}, instead of the $\Theta(n\ln(1/\epsilon))$ iterations required for algorithms that follow the span condition. 

We also improve the lower bound of \parencite{ArjevaniShamir2016_dimensionfreea}
to $\Omega(n+(\frac{n}{1+\p{\ln\p{n/\kappa}}_+}+\sqrt{n\kappa})\ln\p{1/\eps})$ in Section \ref{sec:Optimality}. That is, the complexity $K(\epsilon)$ of a very general class of algorithm that includes all of the above satisfies the lower bound:
\begin{align}
K\p{\epsilon} & =\begin{cases}
\Omega(n+\sqrt{n\kappa}\ln\p{1/\eps}), &\text{ for } n=\cO\p{\kappa},\\
\Omega(n+\frac{n}{1+\p{\ln\p{n/\kappa}}_+}\ln\p{1/\eps}), &\text{ for } \kappa=\cO\p{n}.
\end{cases}
\end{align}
Hence when $\kappa=\cO\p n$ our modified SVRG has optimal complexity, and when $n=\cO\p{\kappa}$, Katyusha is optimal. 

SDCA doesn't quite follow the span assumption. Also the dimension $n$ of the dual space on which the algorithm runs is inherently small in comparison to $k$, the number of iterations. We complete the picture using different arguments, by showing that its complexity is greater than $\Omega(n\ln(1/\epsilon))$ in Section \ref{sec:SDCA}, and hence SDCA doesn't attain this logarithmic speedup. We leave the analysis of accelerated SAGA and accelerated SDCA to future work.

Our results identify a significant obstacle to high performance when $n\gg\kappa$. The speedup that SVRG and SARAH can be modified to attain in this scenario is somewhat accidental since their original purpose was to minimize memory overhead. However, with the knowledge that this assumption is a point of weakness for VR algorithms, future work may more purposefully exploit this to yield better speedups than SVRG and SARAH can currently attain. Though the complexity of SVRG and SARAH can be made optimal to within a constant factor, this factor is somewhat large, and could potentially be reduced substantially.

Having $n\gg\kappa$, which has been referred to as the ``big data
condition'', is rather common: For instance \parencite{RouxSchmidtBach2012_stochastic}
remarks that $\kappa=\sqrt{n}$ is a nearly optimal choice for regularization for empirical risk minimization in some scenarios, \parencite{SridharanShalev-shwartzSrebro2009_fast} considers
$\kappa=\sqrt{n}$, and \parencite{EbertsSteinwart2011_optimal} considers
$\kappa=n^{\beta}$ for $\beta<1$. So for instance, we now have the following corollary (which will follow from Corollary \ref{Cor:OptimalUpperBound} ahead):

\begin{cor} \label{cor:nlogn-to-log}
To obtain accuracy $\epsilon = 1/n^\alpha$ for $\kappa=n^\beta$ and $\alpha,\beta\in\p{0,1}$, modified SVRG requires $\cO\p{n}$ iterations,  whereas algorithms that follow the span assumption require $\cO\p{n\ln\p{n}}$ iterations \parencite{LanZhou2017_optimal} for sufficiently large $d$.
\end{cor}

For large-scale problems, this $\ln\p n$ factor can be
rather large: For instance in the KDD Cup 2012 dataset ($n=149,639,105$ and
$\ln\p n\approx18$), Criteo's Terabyte Click Logs ($n=4,195,197,692$
and $\ln\p n\approx22$), etc. Non-public internal company datasets
can be far larger, with $n$ potentially larger than $10^{15}$. Hence
for large-scale problems in this setting, SVRG, SARAH, and future algorithms designed for the big-data regime can be expected to have much better performance than algorithms following the span condition.

We also analyze Prox-SVRG in the case where  $f_{i}$ are smooth and potentially nonconvex, but the sum $F$ is strongly convex. We build on the work of \parencite{Allen-Zhu2018_katyusha}, which proves state-of-the-art complexity bounds for this setting, and show that we can attain a similar logarithmic speedup without modification. Lower bounds for this context are lacking, so it is unclear if this result can be further improved.

\section{Optimal Convex SVRG} \label{sec:OptimalSVRG}

In this section, we show that the Prox-SVRG algorithm proposed in \parencite{XiaoZhang2014_proximal} for problem \eqref{eq:Average-of-fi} can be sped up by a factor of $\Omega(1+(\ln(n/\kappa))_+)$ when $\kappa=\cO(n)$. A similar speedup is clearly possible for vanilla SVRG and SARAH, which have similar rate expressions. We then refine the lower bound analysis of \parencite{ArjevaniShamir2016_dimensionfreea} to show that the complexity is optimal\footnote{I.e. the complexity cannot be improved among a very broad class of finite-sum algorithms.} when $\kappa=\cO(n)$. Katyusha is optimal in the other scenario when $n=\cO(\kappa)$ by \parencite{ArjevaniShamir2016_iterationa}.

\begin{asmp} \label{assumption 1}
$f_i$ is $L_i-$Lipschitz differentiable for $i=1,2,...,n$. That is, 
    \[
    \|\nabla f_i(x)-\nabla f_i(y)\|\leq L_i\|x-y\|\quad \text{for all} \,\, x,y\in\mathbb{R}^d.
    \]
    $f$ is $L-$Lipschitz differentiable.     $F$ is $\mu-$strongly convex. That is, 
    \[
    F(y)\geq F(x)+\langle \tilde{\nabla}F(x), y-x\rangle+\frac{\mu}{2}\|y-x\|^2 \quad \text{for all}\,\,x,y\in\mathbb{R}^d\,\,\text{and}\,\, \tilde{\nabla}F(x)\in\partial F(x).
    \]

\end{asmp}

\begin{asmp}\label{assumption 2}
\hfill
 $f_i$ is convex for $i=1,2,...,n$; and
     $\psi$ is proper, closed, and convex.

\end{asmp}

\begin{algorithm}[H]    \caption{$\text{Prox-SVRG}(F, x^0, \eta,m)$}\label{euclid}
    \textbf{Input:} $F(x)=\psi(x)+\frac{1}{n}\sum_{i=1}^{n} f_i(x)$, initial vector $x^0$, step size $\eta>0$, number of epochs $K$, probability distribution $P=\cp{p_1,\ldots,p_n}$\\
    \textbf{Output:} vector $x^K$
    \begin{algorithmic}[1]
        \State{$M^k\sim\text{Geom}(\frac{1}{m})$;}
        \For{$k\leftarrow 0,...,K-1$}{}
        \State{$w_0\leftarrow x^k$; $\mu\leftarrow \nabla f(x^k)$;}
        \For{$t \leftarrow 0,...,M^k$}{}
        \State{pick $i_t\in\{1,2,...,n\}\sim P$ randomly;}
        \State{$\tilde{\nabla}_t=\mu+\big(\nabla f_{i_t}(w_t)-\nabla f_{i_t}(w_0)\big)/(np_{i_t});$}
        \State{$w_{t+1}=\argmin_{y\in\mathbb{R}^d}\{\psi(y)+\frac{1}{2\eta}\|y-w_t\|^2+\langle \tilde{\nabla}_t, y\rangle\};$}
        \EndFor
        \State{$x^{k+1}\leftarrow w_{M+1};$}
        \EndFor
    \end{algorithmic}
    \label{alg_1}
\end{algorithm}

We make Assumption \ref{assumption 1} throughout the paper, and Assumption \ref{assumption 2} in this section. Recall the Prox-SVRG algorithm of \parencite{XiaoZhang2014_proximal}, which we reproduce in Algorithm \ref{alg_1}. The algorithm is organized into a series of $K$ \textbf{epochs} of size $M^k$, where $M^k$ is a geometric random variable with success probability $1/m$. Hence epochs have an expected length of $m$. At the start of each epoch, a snapshot $\mu = \nabla f(x^k)$ of the gradient is taken. Then for $M^k$ steps, a random component gradient $\nabla_{i_t}f(w_t)$ is calculated, for an IID random variable $i_t$ with fixed distribution $P$ given by $\PP[i_t=i]=p_i$. This component gradient is used to calculate an unbiased estimate $\tilde{\nabla}_t$ of the true gradient $\nabla f(w_t)$. Each time, this estimate is then used to perform a proximal-gradient-like step with step size $\eta$. At the end of these $M^k$ steps, a new epoch of size $M^{k+1}$ is started, and the process continues.

We first recall a modified Theorem 1 from \parencite{XiaoZhang2014_proximal}. The difference is that in \parencite{XiaoZhang2014_proximal}, the authors used a epoch length of $m$, whereas we use a random epoch length $M^k$ with expectation $\EE M^k = m$. The proof and theorem statement only require only trivial modifications to account for this. This modification is only to unify the different version of SVRG in \parencite{XiaoZhang2014_proximal} and \parencite{Allen-Zhu2018_katyusha}, and makes no difference to the result.

It becomes useful to define the \textbf{effective Lipschitz constant} $L_Q = \text{max}_i L_i/(p_in)$, and the \textbf{effective condition number} $\kappa_Q=L_Q/\mu$ for this algorithm. These reduce to the standard Lipschitz constant $L$, and the standard condition number $\kappa$ in the standard uniform scenario where $L_i=L, \forall i$, and $P$ is uniform.

\begin{thm}[Complexity of Prox-SVRG] \label{upper complexity}
Let Assumptions \ref{assumption 1} and \ref{assumption 2} hold. Then Prox-SVRG defined in Algorithm \ref{alg_1} satisfies:
\begin{align}
    \label{linear convergence}
    \mathbb{E}[F(x^{k})-F(x^*)]&\leq \rho^{k}[F(x^0)-F(x^*)]\\
\text{ for } \rho&=\frac{1+\mu\eta\p{1+4mL_Q\eta}}{\mu \eta m\p{1-4L_Q\eta}}\label{linear rate}
\end{align}

\end{thm}

In previous work, the optimal parameters were not really explored in much detail. In the original paper \parencite{JohnsonZhang2013_accelerating}, the author suggest $\eta=0.1/L$, which results in linear convergence rate $1/4\leq \rho \leq 1/2$ for $m\geq50\kappa$. In \parencite{XiaoZhang2014_proximal}, authors also suggest $\eta=0.1/L$ for $m=100\kappa$, which yields $\rho\approx 5/6$. However, they observe that $\eta=0.01/L$ works nearly as well. In \parencite{NguyenLiuScheinbergTakac2017_sarah}, authors obtain a similar rate expression for SARAH and suggest $\eta=0.5/L$ and $m=4.5\kappa$ which yields $\rho\approx 7/9$. In the following corollary, we propose a choice of $\eta$ and $m$ that leads to an optimal complexity to within a constant factor for $\kappa=\cO (n)$. This result helps explain why the optimal step size observed in prior work appears to be much smaller than the ``standard'' gradient descent step of $1/L$.

\begin{cor} \label{Cor:OptimalUpperBound}
Let the conditions of Theorem \ref{upper complexity} hold, and let $m=n+121 \kappa_Q$, and  $\eta=\kappa_Q^{\frac{1}{2}}m^{-\frac{1}{2}}/(2L_Q)$. The Prox-SVRG in Algorithm \ref{alg_1} has convergence rate $\rho \leq\sqrt{\frac{100}{121+(n/\kappa_Q)}}$, and hence it needs:
%
\begin{align}
K(\epsilon)=\cO\p{\p{\frac{n}{1+(\ln(\frac{n}{\kappa_Q}))_+}+\kappa_Q}\ln{\frac{1}{\epsilon}}+n+\kappa_Q}
\label{upper complexity of Prox-SVRG}
\end{align}
iterations in expectation to obtain a point $x^{K(\epsilon)}$ such that $\EE\sp{f\p{x^{K(\epsilon)}}-f\p{x^*}}<\epsilon$.
\end{cor}

This result is proven in Appendix \ref{App:UpperComplexity}. The $n+\kappa_Q$ term is needed because we assume that at least one epoch is completed. For $n=\cO(\kappa_Q)$, we have a similar convergence rate ($\rho\approx\frac{10}{11}$) and complexity to algorithms that follow the span assumption. For $n\gg\kappa_Q$, we have a convergence rate $\rho\to 0$, and complexity $\cO\p{\frac{n}{1+(\ln(n/\kappa))}\ln(1/\epsilon}$, which can can be much better than $n\ln(1/\epsilon)$. See also Corollary \ref{cor:nlogn-to-log}.

\begin{rem}
In Theorem \ref{upper complexity} and Corollary \ref{Cor:OptimalUpperBound}, the optimal choice of the probability distribution $P=\{p_1,p_2,...,p_n\}$ on $\{1,2,...,n\}$ is $p_i=\frac{L_i}{\sum_{i=1}^n L_j}$ for $i=1,2,...,n$, and $L_Q=\frac{\sum_{i=1}^n L_i}{n}$.
\end{rem}

\subsection{Optimality} \label{sec:Optimality}
The major difference between SAGA, SAG, Miso/Finito, and SDCA without duality, and SVRG and SARAH, is that the former satisfy
what we call the \textbf{span condition} \eqref{eq:SpanCondition}. SVRG, and SARAH, do not, since they also involve full-gradient steps. We refer to SVRG, and SARAH as \textbf{hybrid methods}, since they use full-gradient and partial gradient information to calculate their iterations. We assume for simplicity that $L_i=L$, for all $i$, and that $\psi = 0$. We now present a rewording of Corollary 3 from \parencite{LanZhou2017_optimal}.

\begin{cor}
For every $\eps$ and randomized algorithm on \eqref{eq:Average-of-fi} that follows the span assumption, there are a dimension $d$, and  $L$-smooth, $\mu$-strongly convex functions $f_i$ on $\RR^d$ such that the algorithm takes at least $\Omega\p{\p{n+\sqrt{\kappa n}}\ln\p{1/\epsilon}}$ steps to reach sub-optimality $\EE f\p{x^k}-f\p{x^*}<\eps$.
\end{cor}

The above algorithms that satisfy the span condition all have known upper complexity bound $\cO\p{\p{n+\kappa}\ln\p{1/\eps}}$, and hence for $\kappa =\cO\p{n}$ we have a sharp convergence rate.

However, it turns out that the span assumption is an obstacle to fast convergence when $n\gg\kappa$ (for sufficiently high dimension). In the following theorem, we improve\footnote{Specifically, we improve the analysis of Theorem 2 from this paper.} the analysis of \parencite{ArjevaniShamir2016_dimensionfreea}, to show that the complexity of SVRG obtained in Corollary \ref{Cor:OptimalUpperBound} is optimal to within a constant factor without fundamentally different assumptions on the class of algorithms that are allowed. Clearly this also applies to SARAH. The theorem is actually far more general, and applies to a general class of algorithms called $p-$\textit{CLI oblivious} algorithms introduced in \parencite{ArjevaniShamir2016_dimensionfreea}. This class contains all VR algorithms mentioned in this paper. In Appendix \ref{App:LowerBound}, we give the definition of $p-$CLI oblivious algorithms, as well as the proof of a more general version of Theorem \ref{lower complexity of SVRG, SARAH}.

\begin{thm}[Lower complexity bound of Prox-SVRG and SARAH] \label{lower complexity of SVRG, SARAH}
For all $\mu,L$, there exist $L$-smooth, and $\mu$-strongly convex functions $f_i$ such that at least\footnote{We absorb some smaller low-accuracy terms (high $\epsilon$) as is common practice. Exact lower bound expressions appear in the proof.}
\begin{align}
K\p{\eps} &=\tilde{\Omega}\p{\p{\frac{n}{1+(\ln(\frac{n}{\kappa}))_+}+\sqrt{n\kappa}}\ln{\frac{1}{\epsilon}}+n}
\label{Lower Bound Oblivious PCLI}
\end{align}
iterations are needed for SVRG or SARAH to obtain expected suboptimality $\EE\sp{f\p{K\p{\eps}}-f\p{X^*}}<\eps$.

\end{thm}


\subsection{SDCA} \label{sec:SDCA}
To complete the picture, in the following proposition, which we prove in Appendix \ref{App:lower complexity of SDCA}, we show that SDCA has a complexity lower bound of $\Omega(n\ln(1/\eps))$, and hence attains no logarithmic speedup. SDCA aims to solve the following problem:
\[
\min_{x\in \mathbb{R}^d} F(x)=\frac{1}{n}\sum_{i=1}^n f_i(x)= \frac{1}{n}\sum_{i=1}^n \big(\phi_i(x^Ty_i)+\frac{\lambda}{2}\|x\|^2\big),
\]
where each $y_i\in\mathbb{R}^d$, $\phi_i:\mathbb{R}\rightarrow \mathbb{R}$ is convex and smooth. It does so with coordinate minimization steps on the corresponding dual problem:
\[
\min_{\alpha\in \mathbb{R}^n} D(\alpha)\coloneqq \frac{1}{n}\sum_{i=1}^n \phi^*_i(-\alpha_i)+\frac{\lambda}{2}\|\frac{1}{\lambda n}\sum_{i=1}^n\alpha_i y_i\|^2,
\]
Here $\phi^*_i(u)\coloneqq \max_z\big(zu-\phi_i(z)\big)$ is the convex conjugate of $\phi_i$. Let $i_k$ be an IID sequence of uniform random variables on $\{1,...,n\}$. SDCA updates a dual point $\alpha^k$, while maintaining a corresponding primal vector $x^k$. SDCA can be written as:
\begin{align}\label{eq:SDCA}
    \alpha^{k+1}_i &= \begin{cases}
    \alpha^k_i, &\text{if}\,\,i\neq i_k,\\
    \argmin_{z} D(\alpha^k_1,...,\alpha^k_{i-1}, z, \alpha^k_{i+1},...,\alpha^k_{n}), &\text{if}\,\,i=i_k,
    \end{cases}\\ x^{k+1}&=\frac{1}{n\lambda}\sum_{i=1}^n\alpha^{k+1}_iy_i,
\end{align}
Since SDCA doesn't follow the span assumption, and the number of iterations $k$ is much greater than the dual problem dimension $n$, different arguments to the ones used in \parencite{LanZhou2017_optimal} must be used. Motivated by the analysis in \parencite{ArjevaniShalev-ShwartzShamir2016_lower}, which only proves a lower bound for dual suboptimality, we have the following lower complexity bound, which matches the upper complexity bound given in \parencite{Shalev-ShwartzZhang2013_stochastic} for $\kappa=\cO(n)$.

\begin{prop}[Lower complexity bound of SDCA]
\label{lower complexity of SDCA}
For all $\mu,L,n>2$, there exist $n$ functions $f_i$ that are $L-$smooth, and $\mu-$strongly convex such that 
\begin{align}
K(\epsilon)=\Omega\big(n\ln\frac{1}{\epsilon}\big)
\label{lower bound of SDCA}
\end{align}
iterations are needed for SDCA to obtain expected suboptimality $\mathbb{E}[F(K(\epsilon))-F(x^*)]\leq \epsilon$.
\end{prop}







\section{Why are hybrid methods faster?}
In this section, we explain why SVRG and SARAH, which are
a hybrid between full-gradient and VR methods, are fundamentally faster
than other VR algorithms. We consider the performance of these algorithms on a variation of
the adversarial function example from \parencite{LanZhou2017_optimal,Nesterov2013_introductory}.
The key insight is that the span condition makes this adversarial
example hard to minimize, but that the full gradient steps of SVRG
and SARAH make it easy when $n\gg\kappa$. 

We conduct the analysis in $\ell^{2}$, for simplicity\footnote{This is the Hilbert space of sequence $(x_{i})_{i=1}^\infty$ with $\sum_{i=1}^\infty x_i^2<\infty$}, since the argument
readily applies to $\RR^{d}$. Consider the function introduced in \parencite{Nesterov2013_introductory}
that we introduce for the case $n=1$:
\begin{align*}
\phi\p x & =\frac{L-\sigma}{4}\p{\frac{1}{2}\dotp{x,Ax}-\dotp{e_{1},x}}\text{, for }A=\p{\begin{array}{cccc}
2 & -1\\
-1 & 2 & -1\\
 & -1 & 2 & \ddots\\
 &  & \ddots & \ddots
\end{array}}
\end{align*}
The function $\phi\p x+\frac{1}{2}\sigma\n x^{2}$ is $L$-smooth
and $\sigma$-strongly convex. Its minimizer $x^{*}$ is given by
$\p{q_{1},q_{1}^{2},q_{1}^{3},\ldots}$ for $q_{1}=\p{\kappa^{1/2}-1}/\p{\kappa^{1/2}+1}$.
We assume that $x^{0}=0$ with no loss in generality. Let $N\p x$
be position of the last nonzero in the vector. E.g. $N\p{0,2,3,0,4,0,0,0,\ldots}=5$.
$N\p x$ is a control on how close $x$ can be to the solution. If
$N\p x=N$, then clearly:
\begin{align*}
\n{x-x^{*}}^{2} & \geq\min_{y\text{ s.t. }N\p y=N}\n{y-x^{*}}^{2}=\n{\p{0,\ldots,0,q_{1}^{N+1},q_{1}^{N+2},\ldots}}^{2}=q_{1}^{2N+2}/\p{1-q_{1}^{2}}
\end{align*}
Because of the tridiagonal pattern of nonzeros in the Hessian $\nabla_{x}^{2}\p{\phi\p x+\frac{1}{2}\sigma\n x^{2}}\p y=\frac{L-\sigma}{4}A+\sigma I$,
the last nonzero $N\p{x^{k}}$ of $x^{k}$ can only increase by $1$
per iteration \emph{by any algorithm that satisfies that span condition}
(e.g. gradient descent, accelerated gradient descent, etc.). Hence
since we have $N\p{x^{0}}=0,$ we have $\n{x^{k}-x^{*}}^{2}/\n{x^{0}-x^{*}}^{2}\geq q_{1}^{2k}$.

For the case $n>1$, let the solution vector $x=\p{x_{1},\ldots,x_{n}}$
be split into $n$ coordinate blocks, and hence define: 
\begin{align}
f\p x & =\sum_{i=1}^n\big(\phi\p{x_{i}}+\frac{1}{2}\sigma\n x^{2}\big)\\ &=\sum_{i=1}^n\p{\frac{L-\sigma}{4}\p{\frac{1}{2}\dotp{x_{i},Ax_{i}}-\dotp{e_{1},x_{i}}}+\frac{1}{2}\p{\sigma n}\n{x_{i}}^{2}}\nonumber\\
 & =\sum_{i=1}^n\p{\frac{\p{L-\sigma+\sigma n}-\sigma n}{4}\p{\frac{1}{2}\dotp{x_{i},Ax_{i}}-\dotp{e_{1},x_{i}}}+\frac{1}{2}\p{\sigma n}\n{x_{i}}^{2}}.
 \label{eq:Adversarial-f-rep-2}
\end{align}
$f$ is clearly the sum of $n$ convex $L$-smooth functions $\phi\p{x_{i}}+\frac{1}{2}\sigma\n x^{2}$,
that are $\sigma$-strongly convex.  \eqref{eq:Adversarial-f-rep-2} shows it is $\sigma n$-strongly
convex and $L-\sigma+\sigma n$-smooth with respect to coordinate
$x_{i}$. Hence the minimizer is given by $x_{i}=\p{q_{n},q_{n}^{2},q_{n}^{3},\ldots}$
for $q_{n}=\p{\p{\frac{\kappa-1}{n}+1}^{1/2}-1}/\p{\p{\frac{\kappa-1}{n}+1}^{1/2}+1}$
for all $i$. Similar to before, $\p{N\p{x_{1}},\ldots,N\p{x_{n}}}$
controls how close $x$ can be to $x^{*}$:
\begin{align*}
\frac{\n{x-x^{*}}^{2}}{\n{x^{*}}^{2}} & =\frac{\sum_{i=1}^{n}\n{x_{i}-\p{q_{n},q_{n}^{2},\ldots}}^{2}}{nq_{n}^{2}/\p{1-q_{n}^{2}}}\geq\sum_{i=1}^{n}q_{n}^{2N\p{x_{i}}}/n
\end{align*}
Let $I_{K,i}$ be the number of times that $i_{k}=i$ for $k=0,1,\ldots,K-1$.
For algorithms that satisfy the span assumption, we have $N\p{x_{i}^{k}}\leq I_{k,i}$.
If we assume that $i_{k}$ is uniform, then $I_{K,i}$ is a binomial
random variable of probability $1/n$ and size $k$. Hence:
\begin{align}
\EE\n{x^{k}-x^{*}}^{2}/\n{x^{0}-x^{*}}^{2} & \geq\EE\sum_{i=1}^{n}q_{n}^{2N\p{x_{i}^{k}}}/n\geq\EE\sum_{i=1}^{n}q_{n}^{2I_{k,i}}/n\nonumber\\
 & =\EE q_{n}^{2I_{k,i}}=\p{1-n^{-1}\p{1-q_{n}^{2}}}^{k}\label{eq:Binomial-RV}\\
 & \geq\p{1-4n^{-1}/\p{\p{\frac{\kappa-1}{n}+1}^{1/2}+1}}^{k}\nonumber\\
 & \geq\p{1-2n^{-1}}^{k}\nonumber
\end{align}
for $n\geq\kappa$. the second equality in \eqref{eq:Binomial-RV} follows from the factor that $I_{i,k}$ is a binomial random variable. Hence after 1 epoch, $\EE\n{x^{k}-x^{*}}^{2}$
decreases by a factor of at most $\approx e^{2}$, whereas for SVRG it decreases
by at least a factor of $\sim\p{n/\kappa}^{1/2}$, which is $\gg e^{2}$ for
$n\gg\kappa$. To help understand why, consider trying the above analysis
on SVRG for 1 epoch of size $n$. Because of the full-gradient step,
we actually have $N\p{w_{i}^{n}}\leq1+I_{n,i}$, and hence:
\begin{align*}
\EE\n{w^{n}-x^{*}}^{2}/\n{x^{0}-x^{*}}^{2} & \geq\EE\sum_{i=1}^{n}q_{n}^{2\p{I_{n,1}+1}}\geq q_{n}^{2}\p{1-2n^{-1}}^{n}
 \approx\p{\frac{1}{4}\frac{\kappa-1}{n}}^{2}e^{-2}
\end{align*}
Hence attempting the above results in a much smaller lower bound.

What it comes down to is that when $n\gg\kappa$, we have $\EE\sum_{i=1}^{n}q_{n}^{2I_{i,k}}/n\gg\EE\sum_{i=1}^{n}q_{n}^{2\EE I_{i,k}}/n$.
The interpretation is that for this objective, the progress towards
a solution is limited by the component function $f_{i}$ that is
minimized the least. The full gradient step ensures that at least
some progress is made toward minimizing every $f_{i}$. For algorithms
that follow the span assumption, there will invariably be many indices
$i$ for which no gradient $\nabla f_{i}$ is calculated, and hence
$x_{i}^{k}$ can make no progress towards the minimum. This may be related to the observation that sampling without replacement can often speed up randomized algorithms. However, on the other hand, it
is well known that full gradient methods fail to achieve a good convergence
rate for other objectives with the same parameters $\mu,L,n$ (e.g. $f\p x=\frac{1}{n}\sum_{i=1}^n\phi\p x+\frac{1}{2}\mu\n x^{2}$). Hence we conclude that it is because SVRG combines both full-gradient and
VR elements that it is able to outperform both VR and full-gradient
algorithms.

\section{Nonconvex Prox-SVRG}
In this section, we show that when $f_i$ is merely assumed to be $L_i$ smooth and possibly nonconvex, there is also a logarithmic speedup. This is based on the analysis of Prox-SVRG found in \parencite{Allen-Zhu2018_katyusha}. The proof of Theorem \ref{nonvex upper complexity} can be found in Appendix \ref{App:Nonconvex-SVRG}.

\begin{thm} \label{nonvex upper complexity}
Under Assumption \ref{assumption 1}, let $x^*=\argmin_x F(x)$,  $\overline{L}=(\sum_{i=1}^n\frac{L_i^2}{n^2p_i})^{\frac{1}{2}}$,  $\kappa=\frac{L}{\mu}$, and $\eta=\frac{1}{2}\min\{\frac{1}{L}, (\frac{1}{\overline{L}^2m})^{\frac{1}{2}}\}$. Then the Prox-SVRG in Algorithm \ref{alg_1} satisfies:
\begin{align}
    \label{nonconvex linear convergence}
    \mathbb{E}[F({x}^{k})-F(x^*)]&\leq \cO(\rho^{k})[F(x^0)-F(x^*)],\\
    \text{ for }\rho&=\frac{1}{1+\frac{1}{2}m\eta\mu}\label{nonconvex linear rate}.
\end{align}
Hence for $m=\min\{n,2\}$, in order to obtain an $\epsilon$-optimal solution in terms of function value, the SVRG in Algorithm \ref{alg_1} needs at most 
\begin{align}
\label{nonconvex K}
K=\cO\big((\frac{n}{\ln{(1+\frac{n}{4\kappa})}}+\frac{n}{\ln(1+(\frac{n\mu^2}{4\overline{L}^2})^{1/2})}+\kappa+\sqrt{n}\frac{\Bar{L}}{\mu})\ln\frac{1}{\epsilon}\big)+2n
\end{align}
gradient evaluations in expectation.

\end{thm}

The complexity of nonconvex SVRG using the original analysis of \parencite{Allen-Zhu2017_katyusha} would have been
\begin{align}
K=\cO\big((n+\kappa+\sqrt{n}\frac{\Bar{L}}{\mu})\ln\frac{1}{\epsilon}\big) 
\end{align}
Hence we have obtained a similar logarithmic speedup as we obtained in Corollary \ref{Cor:OptimalUpperBound}.

\begin{rem}
In Theorem \ref{nonvex upper complexity}, the optimal choice of the probability distribution $P=\{p_1,p_2,...,p_n\}$ on $\{1,2,...,n\}$ is $p_i=\frac{L_i^2}{\sum_{i=1}^n L_j^2}$ for $i=1,2,...,n$, and $\overline{L}=(\frac{\sum_{i=1}^n L_i^2}{n})^{\frac{1}{2}}$.
\end{rem}


\pdfbookmark[1]{References}{References}
\printbibliography

\newpage

\appendix

\section{Upper Complexity Bound for Convex SVRG} \label{App:UpperComplexity}

\begin{proof}[Proof of Theorem \ref{upper complexity} and Corollary \ref{Cor:OptimalUpperBound}]
\eqref{linear convergence} and \eqref{linear rate} follows directly from the analysis of \parencite[Thm 3.1]{XiaoZhang2014_proximal} with slight modification.

For the linear rate $\rho$ in \eqref{linear rate}, we have
\begin{align*}
    \rho&\overset{(\mathrm{a})}{\leq} 2(\frac{1}{\mu\eta m}+4L_Q\eta+\frac{1}{m})\\
    &\overset{(\mathrm{b})}{=} 2(\frac{1}{\mu\eta m}+2\kappa_Q^{\frac{1}{2}}m^{-\frac{1}{2}})+\frac{2}{m}\\
    &\overset{(\mathrm{c})}{=} 2\Big(\frac{1}{\mu m}2L_Q\kappa_Q^{-\frac{1}{2}}m^{\frac{1}{2}}+2\kappa_Q^{\frac{1}{2}}m^{-\frac{1}{2}}\Big)+\frac{2}{m}\\
    &=8\kappa_Q^{\frac{1}{2}}m^{-\frac{1}{2}}+\frac{2}{m}\\
    &\overset{(\mathrm{d})}{\leq}8\kappa_Q^{\frac{1}{2}}m^{-\frac{1}{2}}+2\kappa_Q^{\frac{1}{2}}m^{-\frac{1}{2}}\\
    &=10\kappa_Q^{\frac{1}{2}}m^{-\frac{1}{2}},
\end{align*}
where (a) is by $\eta=\frac{\kappa_Q^{\frac{1}{2}}m^{-\frac{1}{2}}}{2L_Q}\leq\frac{1}{22L_Q}\leq \frac{1}{8L_Q}$, (b) is by $\eta = \frac{\kappa_Q^{\frac{1}{2}}m^{-\frac{1}{2}}}{2L_Q}$, (c) is by $\frac{1}{\eta}= 2L_Qm^{\frac{1}{2}}\kappa_Q^{-\frac{1}{2}}$, and (d) follows from $\kappa_Q^{\frac{1}{2}}m^{\frac{1}{2}}\geq 1$.

Therefore, the epoch complexity (i.e. the number of epochs required to reduce the suboptimality to below $\epsilon$) is
\begin{align*}
    K_0&=\lceil\frac{1}{\ln(\frac{1}{10}m^{\frac{1}{2}}\kappa_Q^{-\frac{1}{2}})}\ln\frac{F(x^0)-F(x^*)}{\epsilon}\rceil\\
    &\leq \frac{1}{\ln(\frac{1}{10}m^{\frac{1}{2}}\kappa_Q^{-\frac{1}{2}})}\ln\frac{F(x^0)-F(x^*)}{\epsilon}+1\\
    &=\frac{2}{\ln(1.21+\frac{1}{100}\frac{n}{\kappa_Q})}\ln\frac{F(x^0)-F(x^*)}{\epsilon}+1\\
    &=\cO\big(\frac{1}{\ln(1.21+
    \frac{n}{100\kappa_Q})}\ln\frac{1}{\epsilon}\big)+1
\end{align*}
where $\lceil \cdot \rceil$ is the ceiling function, and the second equality is due to $m=n+121\kappa_Q$.

Hence, the gradient complexity is 
\begin{align*}
    K&=(n+m)K_0\\
    &\leq \cO\big(\frac{n+\kappa_Q}{\ln
    (1.21+\frac{n}{100\kappa_Q})}\ln\frac{1}{\epsilon}\big)+n+121\kappa_Q,
\end{align*}
which is equivalent to \eqref{upper complexity of Prox-SVRG}.
\end{proof}

\section{Lower Complexity Bound for Convex SVRG} \label{App:LowerBound}

\begin{defn}\parencite[Def. 2]{ArjevaniShamir2016_dimensionfreea}
An optimization algorithm is called a Canonical Linear Iterative (CLI) optimization
algorithm, if given a function $F$ and initialization
points $\{w^{0}_i\}_{i\in J}$, where $J$ is some index set, it operates by iteratively generating points such
that for any $i\in J$,
\[
w^{k+1}_i = \sum_{j\in J} O_F(w^k_j; \theta^k_{ij}),\quad k=0,1,...
\]
holds, where $\theta^k_{ij} $ are parameters chosen, stochastically or deterministically, by the algorithm, possibly depending on the side-information. $O_F$ is an oracle parameterized by $\theta^k_{ij}$. If the parameters do not depend on previously acquired oracle answers, we say that the given algorithm is oblivious. Lastly, algorithms with $|J|\leq p$, for some $p\in \mathbb{N}$, are denoted by p-CLI.
\end{defn}
In \parencite{ArjevaniShamir2016_dimensionfreea}, two types of oblivious oracles are considered. The generalized first order oracle for $F(x)=\frac{1}{n}\sum_{i=1}^n f_i(x)$
\[
O(w; A, B, C, j)=A\nabla f_j(w)+Bw+C, \quad A,B\in\mathbb{R}^{d\times d}, C\in \mathbb{R}^d, j\in [n].
\]
The steepest coordinate descent oracle for $F(x)=\frac{1}{n}\sum_{i=1}^n f_i(x)$ is given by
\[
O(w;i,j)=w+t^*e_i, \quad t^*\in\argmin_{t\in \mathbb{R}}f_j(w_1,...,w_{i-1}, w+t, w_{i+1},...,w_d), j\in [n],
\]
where $e_i$ is the $i$th unit vector. SDCA, SAG, SAGA, SVRG, SARAH, etc.  without proximal terms are all $p-$CLI oblivious algorithms.


We now state the full version of Theorem \ref{lower complexity of SVRG, SARAH}.

\begin{thm}[Lower complexity bound oblivious p-CLI algorithms] \label{lower complexity OPCLI}
For any oblivious p-CLI algorithm $A$, for all $\mu,L,k$, there exist $L$-smooth, and $\mu$-strongly convex functions $f_i$ such that at least\footnote{We absorb some smaller low-accuracy terms (high $\epsilon$) as is common practice. Exact lower bound expressions appear in the proof.}: 
\begin{align}
K\p{\eps} &=\tilde{\Omega}\p{\p{\frac{n}{1+(\ln(\frac{n}{\kappa}))_+}+\sqrt{n\kappa}}\ln{\frac{1}{\epsilon}}+n}
\label{Lower Bound Oblivious PCLI append}
\end{align}
iterations are needed for $A$ to obtain expected suboptimality $\EE\sp{f\p{K\p{\eps}}-f\p{X^*}}<\eps$.

\end{thm}

\begin{proof}[Proof of Theorem \ref{lower complexity OPCLI}]


In this proof, we use lower bound given in \parencite[Thm 2]{ArjevaniShamir2016_dimensionfreea}, and refine its proof for the case $n\geq \frac{1}{3}\kappa$.

\parencite[Thm 2]{ArjevaniShamir2016_dimensionfreea} gives the following lower bound,

\begin{align}
    K(\epsilon)\geq \Omega(n+\sqrt{n(\kappa-1)}\ln{\frac{1}{\epsilon}}).\label{case 1}
\end{align}

Some smaller low-accuracy terms are absorbed are ignored, as is done in \parencite{ArjevaniShamir2016_dimensionfreea}. For the case $n\geq \frac{1}{3}\kappa$, the proof of \parencite[Thm 2]{ArjevaniShamir2016_dimensionfreea} tells us that, for any $k\geq 1$, there exist $L-$Lipschitz differentiable and $\mu-$strongly convex quadratic functions $f^k_1, f^k_2,...,f^k_n$ and $F^k=\frac{1}{n}\sum_{i=1}^n f^k_i$, such that  for any $x^0$, the $x^K$ produced after $K$ gradient evaluations, we have\footnote{note that for the SVRG in Algorithm \ref{alg_1} with $\psi=0$, each update in line $7$ is regarded as an iteration.}
\[
\mathbb{E}[F^K(x^K)-F^K(x^*)]\geq \frac{\mu}{4}(\frac{nR\mu}{L-\mu})^2(\frac{\sqrt{1+\frac{\kappa-1}{n}}-1}{\sqrt{1+\frac{\kappa-1}{n}}+1})^{\frac{2K}{n}},
\]
where $R$ is a constant and $\kappa=\frac{L}{\mu}$. 

Therefore, in order for $\epsilon\geq\mathbb{E}[F(x^k)-F(x^*)]$, we must have
\[
\epsilon \geq \frac{\mu}{4}(\frac{nR\mu}{L-\mu})^2(\frac{\sqrt{1+\frac{\kappa-1}{n}}-1}{\sqrt{1+\frac{\kappa-1}{n}}+1})^{\frac{2K}{n}}=\frac{\mu}{4}(\frac{nR\mu}{L-\mu})^2(1-\frac{2}{1+\sqrt{1+\frac{\kappa-1}{n}}})^{\frac{2k}{n}}.
\]
Since $1+\frac{1}{3}x\leq \sqrt{1+x}$ when $0\leq x \leq 3$, and $0\leq \frac{\kappa-1}{n}\leq \frac{\kappa}{n}\leq 3$, we have
\[
\epsilon\geq \frac{\mu}{4}(\frac{nR\mu}{L-\mu})^2(1-\frac{2}{2+\frac{1}{3}\frac{\kappa-1}{n}})^{\frac{2K}{n}},
\]
or equivalently,
\[
K\geq \frac{n}{2\ln(1+\frac{6n}{\kappa-1})}\ln\big(\frac{\frac{\mu}{4}(\frac{nR}{\kappa-1})^2}{\epsilon}\big).
\]
As a result,
\begin{align*}
    K&\geq\frac{n}{2\ln(1+\frac{6n}{\kappa-1})}\ln\frac{1}{\epsilon}+\frac{n}{2\ln(1+\frac{6n}{\kappa-1})}\ln\big(\frac{\mu}{4}(\frac{nR}{\kappa-1})^2\big)\\
    &=\frac{n}{2\ln(1+\frac{6n}{\kappa-1})}\ln\frac{1}{\epsilon}+\frac{n}{2\ln(1+\frac{6n}{\kappa-1})}\ln(\frac{\mu R^2}{24})+\frac{n}{\ln(1+\frac{6n}{\kappa-1})}\ln\frac{6n}{\kappa-1}.
\end{align*}
Since $\frac{\ln\frac{6n}{\kappa-1}}{\ln(1+\frac{6n}{\kappa-1})}\geq \frac{\ln2}{\ln 3}$ when $\frac{n}{\kappa -1}\geq \frac{n}{\kappa}\geq \frac{1}{3}$, for small $\epsilon$ we have
\begin{align}
    K&\geq \frac{n}{2\ln(1+\frac{6n}{\kappa-1})}\ln\frac{1}{\epsilon}+\frac{n}{2\ln(1+\frac{6n}{\kappa-1})}\ln(\frac{\mu R^2}{24})+\frac{\ln 2}{\ln 3}n\nonumber\\
    &=\Omega\big(\frac{n}{\ln(1+\frac{6n}{\kappa-1})}\ln\frac{1}{\epsilon}\big)+\frac{\ln 2}{\ln 3}n\\
    &= \Omega(\frac{n}{1+(\ln(n/\kappa))_+}\ln(1/\epsilon)+n)\label{case 2}
\end{align}
Now the expression in \eqref{case 2} is valid for $n\geq\frac{1}{3}\kappa$. When $n<\frac{1}{3}\kappa$, the lower bound in \eqref{case 2} is asymptotically equal to $\Omega(n\ln(1/\epsilon)+n)$, which is dominated by \eqref{case 1}. Hence the lower bound in \eqref{case 2} is valid for all $\kappa,n$. 

We may sum the lower bounds in \eqref{case 1} and \eqref{case 2} to obtain \eqref{Lower Bound Oblivious PCLI append}. This is because given an oblivious p-CLI algorithm, we may simply chose the adversarial example that has the corresponding greater lower bound.
\end{proof}

\section{Lower Complexity Bound for SDCA}\label{App:lower complexity of SDCA}
\begin{proof}[Proof of Propsition \ref{lower complexity of SDCA}]
Let $\phi_i(t)=\frac{1}{2}t^2$, $\lambda=\mu$, and $y_i$ be the $i$th column of $Y$, where $Y=c(n^2I+J)$ and $J$ is the matrix with all elements being $1$, and $c=(n^4+2n^2+n)^{-1/2}(L-\mu)^{1/2}$. Then 
\begin{align*}
f_i(x)&=\frac{1}{2}(x^Ty_i)^2+\frac{1}{2}\mu\|x\|^2,
\\
F(x)&=\frac{1}{2n}\|Y^Tx\|^2+\frac{1}{2}\mu\|x\|^2,\\
D(\alpha)&=\frac{1}{n\mu}(\frac{1}{2n}\|Y\alpha\|^2+\frac{1}{2}\mu\|\alpha\|^2).
\end{align*}
Since 
\[
\|y_i\|^2=c^2\big((n^2+1)^2+n-1\big)=c^2(n^4+2n^2+n)=L-\mu,
\]
$f_i$ is $L-$smooth and $\mu-$strongly convex, and that $x^*=\mathbf{0}$.

We also have
\[
\nabla D(\alpha)=\frac{1}{n\mu}(\frac{1}{n}Y^2\alpha+\mu\alpha)=\frac{1}{n\mu}\big((c^2n^3I+2nc^2J+c^2J)\alpha +\mu\alpha\big),
\]
So for every $k\geq 0$, minimizing with respect to $\alpha_{i_k}$ as in \eqref{eq:SDCA} yields the optimality condition:
\begin{align*}
0&=e_{i_k}^T \nabla D(\alpha^{k+1})\\
&=\frac{1}{n\mu}\big(c^2n^3\alpha^{k+1}_{i_k}+2c^2n(\sum_{j\neq i_k} \alpha^{k}_j+\alpha^{k+1}_{i_k})+c^2(\sum_{j\neq i_k} \alpha^{k}_j+\alpha^{k+1}_{i_k})+\mu \alpha^{k+1}_{i_k}\big).
\end{align*}
Therefore, rearranging yields:
\[
\alpha^{k+1}_{i_k}=-\frac{(c^2+2c^2n)}{c^2n^3+2c^2n+c^2+\mu}\sum_{j\neq i_k}\alpha^k_j=-\frac{(c^2+2c^2n)}{c^2n^3+2c^2n+c^2+\mu}(e_{i_k}^T(J-I)\alpha^k).
\]
As a result,
\[
\alpha^{k+1}=(I-e_{i_k}e_{i_k}^T)\alpha^k-\frac{(c^2+2c^2n)}{c^2n^3+2c^2n+c^2+\mu}(e_{i_k}e_{i_k}^T(J-I)\alpha^k).
\]
Taking full expectation on both sides gives
\[
\mathbb{E}\alpha^{k+1}=\Big((1-\frac{1}{n})I-\frac{(c^2+2c^2n)}{c^2n^3+2c^2n+c^2+\mu}\frac{J-I}{n}\Big)\mathbb{E}\alpha^k\triangleq T\EE\alpha^k.
\]
for linear operator $T$. Hence we have by Jensen's inequality:
\begin{align*}
\EE\n{x^{k}}^{2} & =n^{-2}\mu^{-2}\EE\n{Y\alpha^{k}}^{2}\\
 & \geq n^{-2}\mu^{-2}\n{Y\EE\alpha^{k}}^{2}\\
 & =n^{-2}\mu^{-2}\n{YT^{k}\alpha^{0}}^{2}
\end{align*}
We let $\alpha^{0}=\p{1,\ldots,1}$,
which is an vector of $T$. Let us say the corresponding eigenvalue
for $T$ is $\theta$:
\begin{align}
\EE\n{x^{k}}^{2} & \geq\theta^{2k}n^{-2}\mu^{-2}\n{Y\alpha^{0}}^{2}\\
 & =\theta^{2k}\n{x^{0}}^{2}\label{eq:Theta-bound}
\end{align}
We now analyze the value of $\theta$:
\begin{align*}
 \theta & =(1-\frac{1}{n})-\frac{(c^{2}+2c^{2}n)}{c^{2}n^{3}+2c^{2}n+c^{2}+\mu}\frac{n-1}{n}\\
 & =1-\frac{1}{n}-\frac{1+2n}{n^{3}+2n+1+\mu c^{-2}}\frac{n-1}{n}\\
 & \geq1-\frac{1}{n}-\frac{1+2n}{n^{3}+2n+1}\\
 & \geq1-\frac{2}{n}
\end{align*}
for $n>2$. This in combination with \eqref{eq:Theta-bound} yields \eqref{lower bound of SDCA}.
\end{proof}

\section{Nonconvex SVRG Analysis}\label{App:Nonconvex-SVRG}

\begin{proof}[Proof of Theorem \ref{nonvex upper complexity}]
Without loss of generality, we can assume $x^*=\mathbf{0}$ and $F(x^*)=0.$

According to lemma 3.3 and Lemma 5.1 of \parencite{Allen-Zhu2018_katyusha}, for any $u\in\RR^d$, and $\eta\leq\frac{1}{2}\min\cp{\frac{1}{L},\frac{1}{\sqrt{m}\bar{L}}}$ we have
\[
\mathbb{E}[F(x^{j+1})-F(u))]\leq \mathbb{E}[-\frac{1}{4m\eta}\|x^{j+1}-x^j\|^2+\frac{\langle x^j-x^{j+1}, x^j-u\rangle}{m\eta}-\frac{\mu}{4}\|x^{j+1}-u\|^2],
\]
or equivalently,
\[
\mathbb{E}[F(x^{j+1})-F(u))]\leq \mathbb{E}[\frac{1}{4m\eta}\|x^{j+1}-x^j\|^2+\frac{1}{2m\eta}\|x^j-u\|^2-\frac{1}{2m\eta}\|x^{j+1}-u\|^2-\frac{\mu}{4}\|x^{j+1}-u\|^2].
\]
Setting $u=x^*=0$ and $u=x^j$ yields the following two inequalities:
\begin{align}
    F(x^{j+1})&\leq \frac{1}{4m\eta}(\|x^{j+1}-x^j\|^2+2\|x^j\|^2-2(1+\frac{1}{2}m\eta\mu)\|x^{j+1}\|^2),\label{1}\\
    F(x^{j+1})-F(x^j)&\leq-\frac{1}{4m\eta}(1+m\eta\mu)\|x^{j+1}-x^j\|^2.\label{2}
\end{align}
Define $\tau=\frac{1}{2}m\eta\mu$, multiply $(1+2\tau)$ to \eqref{1}, then add it to \eqref{2} yields
\[
2(1+\tau)F(x^{j+1})-F(x^j)\leq \frac{1}{2m\eta}(1+2\tau)\big(\|x^j\|^2-(1+\tau)\|x^{j+1}\|\big).
\]
Multiplying both sides by $(1+\tau)^j$ gives
\[
2(1+\tau)^{j+1}F(x^{j+1})-(1+\tau)^jF(x^j)\leq \frac{1}{2m\eta}(1+2\tau)\big((1+\tau)^j\|x^j\|^2-(1+\tau)^{j+1}\|x^{j+1}\|\big).
\]
Summing over $j=0,1,...,k-1$, we have
\[
(1+\tau)^k F(x^k)+\sum_{j=0}^{k-1}(1+\tau)^j F(x^j)-F(x^0)\leq \frac{1}{2m\eta}(1+2\tau)(\|x^0\|^2-(1+\tau)^k\|x^k\|^2).
\]
Since $F(x^j)\geq 0$, we have
\[
F({x}^k)(1+\tau)^k \leq F(x^0)+\frac{1}{2m\eta}(1+2\tau)\|x^0\|^2.
\]
By the strong convex of $F$, we have $F(x^0)\geq \frac{\mu}{2}\|x^0\|^2$, therefore
\[
F({x}^k)(1+\tau)^k\leq F(x^0)(2+\frac{1}{2\tau}),
\]
Finally, $\eta=\frac{1}{2}\min\{\frac{1}{L}, (\frac{1}{\overline{L}^2m})^{\frac{1}{2}}\}$ gives 
\[
\frac{1}{\tau}=4\max\{\frac{\kappa}{m}, (\frac{\overline{L}^2}{m\mu^2})^{\frac{1}{2}}\}\leq 4(\frac{\kappa}{m}+(\frac{\overline{L}^2}{m\mu^2})^{-\frac{1}{2}}),
\]
which yields
\[
F(x^k)\leq (1+\tau)^{-k}F(x^0)\big(2+2(\frac{\kappa}{m}+(\frac{\overline{L}^2}{m\mu^2})^{-\frac{1}{2}})\big).
\]
To prove \eqref{nonconvex linear rate}, we notice that
\[
\tau=\frac{1}{4}\min\{\frac{m}{\kappa}, (\frac{m\mu^2}{\overline{L}^2})^{\frac{1}{2}}\}, 
\]
so we have
\[
\frac{1}{\ln(1+\tau)}\leq \frac{1}{\ln(1+\frac{m}{4\kappa})}+\frac{1}{\ln\big(1+(\frac{m\mu^2}{4\overline{L}})^{\frac{1}{2}}\big)}
\]
Now for small $\epsilon$, the epoch complexity can be written as
\begin{align*}
K_0&=\lceil\frac{1}{\ln(1+\tau)}\ln\frac{F(x^0)(2+2(\frac{\kappa}{m}+(\frac{\overline{L}^2}{m\mu^2})^{-\frac{1}{2}}))}{\epsilon}\rceil\\
&\leq\cO\Big((\frac{1}{\ln(1+\frac{m}{4\kappa})}+\frac{1}{\ln\big(1+(\frac{m\mu^2}{4\overline{L}})^{\frac{1}{2}}\big)})\ln\frac{1}{\epsilon}\Big)+1.
\end{align*}
Since $m=\min\{2,n\}$, we have a gradient complexity of
\[
K=(n+m)K_0\leq\cO\Big((\frac{n}{\ln(1+\frac{n}{4\kappa})}+\frac{n}{\ln\big(1+(\frac{n\mu^2}{4\overline{L}})^{\frac{1}{2}}\big)})\ln\frac{1}{\epsilon}\Big)+2n.
\]
And this is equivalent to the expression in \eqref{nonconvex K}.

\end{proof}

\end{document}

%% file: FancyArticlePackages.tex
\usepackage[T1]{fontenc}
\usepackage[utf8]{inputenc}

\usepackage{lmodern}

\usepackage{bm}

\usepackage{color}

\usepackage{microtype}

\usepackage{amsmath}

\usepackage{amsthm}

\usepackage{amssymb}

\usepackage[showonlyrefs]{mathtools}

\usepackage{array}

\usepackage{authblk}

\usepackage{geometry}

\usepackage{fancyhdr, blindtext}
\newcommand{\changefont}{\fontsize{7}{9}\selectfont}
\newcolumntype{L}[1]{>{\raggedright\let\newline\\\arraybackslash\hspace{0pt}}m{#1}}
\newcolumntype{C}[1]{>{\centering\let\newline\\\arraybackslash\hspace{0pt}}m{#1}}
\newcolumntype{R}[1]{>{\raggedleft\let\newline\\\arraybackslash\hspace{0pt}}m{#1}}

\usepackage[colorinlistoftodos]{todonotes}

\usepackage{hyperref}
\hypersetup{colorlinks=true, linkcolor=blue, citecolor=blue, urlcolor=blue,bookmarksnumbered,bookmarksdepth=3,breaklinks=true}


%% file: FancyMacros.tex
\usepackage{mathtools}
\usepackage[normalem]{ulem} 

\usepackage{amsmath}
\usepackage{amssymb}
\usepackage{mathtools}
\usepackage{mathrsfs}

\newcommand{\cut}[1]{{}}

\newcommand{\EE}{{\mathbb{E}}} 				
\newcommand{\PP}{\mathbb{P}} 				
\newcommand{\RR}{\mathbb{R}}				

\DeclareMathOperator*{\argmin}{arg\,min}

\newcommand{\DeclareAutoPairedDelimiter}[3]{%
	\expandafter\DeclarePairedDelimiter\csname Auto\string#1\endcsname{#2}{#3}%
	\begingroup\edef\x{\endgroup
		\noexpand\DeclareRobustCommand{\noexpand#1}{%
			\expandafter\noexpand\csname Auto\string#1\endcsname*}}%
	\x}

\DeclareAutoPairedDelimiter{\p}{(}{)} 					
\DeclareAutoPairedDelimiter{\sp}{[}{]} 					
\DeclareAutoPairedDelimiter{\abs}{|}{|} 					
\DeclareAutoPairedDelimiter{\cp}{\{}{\}} 				
\DeclareAutoPairedDelimiter{\dotp}{\langle}{\rangle} 	
\DeclareAutoPairedDelimiter{\n}{\Vert}{\Vert} 			
\DeclareAutoPairedDelimiter{\cl}{\lceil}{\rceil}

\newcommand{\cO}{{\mathcal{O}}}



%% file: FancyMathAndTheoremEnvironments.tex
\usepackage{amsmath}
\usepackage{amsthm}
\usepackage{amssymb}
\usepackage{mathtools}
\usepackage{mathrsfs}

\usepackage[nameinlink]{cleveref}

\newcommand{\bc}{\begin{center}}
\newcommand{\ec}{\end{center}}

\newcommand{\bdm}{\begin{displaymath}}
\newcommand{\edm}{\end{displaymath}}

\newcommand{\beq}{\begin{equation}}
\newcommand{\eeq}{\end{equation}}

\newcommand{\bfl}{\begin{flushleft}}
\newcommand{\efl}{\end{flushleft}}

\newcommand{\bt}{\begin{tabbing}}
\newcommand{\et}{\end{tabbing}}

\newcommand{\beqn}{\begin{align}}
\newcommand{\eeqn}{\end{align}}

\newcommand{\beqs}{\begin{align*}} 
\newcommand{\eeqs}{\end{align*}}  

\newtheoremstyle{Fancyplain}
{\topsep}   
{\topsep}   
{\itshape}  
{0pt}       
{\bfseries} 
{}         
{5pt plus 1pt minus 1pt} 
{\thmname{#1} \thmnumber{#2}. \thmnote{\normalfont\bfseries#3.}}

\theoremstyle{Fancyplain}
\newtheorem{thm}{Theorem}
\newtheorem{lem}{Lemma}
\newtheorem{cor}[lem]{Corollary}
\newtheorem{prop}[lem]{Proposition}

\crefname{thm}{Thm.}{Thms.}
\Crefname{thm}{Theorem}{Theorems}
\crefname{lem}{Lem.}{Lems.}
\Crefname{lem}{Lemma}{Lemmas}
\crefname{cor}{Cor.}{Cors.}
\Crefname{cor}{Corollary}{Corollaries}
\crefname{prop}{Prop.}{Props.}
\Crefname{prop}{Proposition}{Propositions}

\newtheoremstyle{Fancydefinition}
{\topsep}   
{\topsep}   
{\normalfont}  
{0pt}       
{\bfseries} 
{}         
{5pt plus 1pt minus 1pt} 
{\thmname{#1} \thmnumber{#2}. \thmnote{\normalfont\bfseries#3.}}

\theoremstyle{Fancydefinition}
\newtheorem{defn}[lem]{Definition}

\newtheorem{rem}{Remark}
\newtheorem{asmp}{Assumption}

\crefname{defn}{Defn.}{Defns.}
\Crefname{defn}{Definition}{Definitions}
\crefname{example}{Ex.}{Exs.}
\Crefname{example}{Example}{Examples}
\crefname{xca}{Ex.}{Exs.}
\Crefname{xca}{Exercise}{Exercises}
\crefname{rem}{Rem.}{Rems.}
\Crefname{rem}{Remark}{Remarks}
\crefname{asmp}{Asmp.}{Asmps.}
\Crefname{asmp}{Assumption}{Assumptions}
\crefname{section}{Sec.}{Secs.}
\Crefname{section}{Section}{Sections}

\numberwithin{equation}{section}
\numberwithin{figure}{section}